\documentclass{article}
\usepackage[utf8]{inputenc}
\usepackage{srcltx}
\usepackage{amssymb,amsmath,amsfonts,amsthm}
\usepackage{cite}
\usepackage{verbatim}

\usepackage[english]{babel}
\usepackage{amsmath}
\usepackage{longtable}
\usepackage{tikz}
\usepackage{amssymb}
\usepackage{amsfonts}
\usepackage{textcomp}

\newtheorem{lemma}{Lemma}

\newtheorem{theorem}{Theorem}

\newtheorem{corollary}{Corollary}
\newtheorem{pr}{Proposition}

\newcommand{\Aut}{\operatorname{Aut}}
\newcommand{\Out}{\operatorname{Out}}
\newcommand{\prk}{\operatorname{prk}}

\title{On groups having the prime graph as alternating and symmetric groups}
\author{Ilya Gorshkov and Alexey Staroletov\footnote{The work was supported by the program of fundamental scientific researches of the SB RAS \textnumero I.1.1., project \textnumero 0314-2016-0001}}
\date{\vspace{-5ex}}

\begin{document}
\newcommand{\Addresses}{{% additional braces for segregating \footnotesize
  \bigskip
  \footnotesize

  I.~Gorshkov, \textsc{Universidade Federal do ABC, Santo André, São Paulo, Brazil;}\par\nopagebreak
  \textsc{Sobolev Institute of Mathematics, Novosibirsk, Russia}\par\nopagebreak
  \textit{E-mail address: } \texttt{ilygor8@gmail.com}

  \medskip

  A.~Staroletov, \textsc{Sobolev Institute of Mathematics, Novosibirsk, Russia;}\par\nopagebreak
  \textsc{Novosibirsk State University, Novosibirsk, Russia}\par\nopagebreak
  \textit{E-mail address: } \texttt{staroletov@math.nsc.ru}
}}

\maketitle
\begin{abstract}
The {\it prime graph} $\Gamma(G)$ of a finite group $G$ is the graph
whose vertex set is the set of prime divisors of $|G|$ and in which
two distinct vertices $r$ and $s$ are adjacent if and only if there exists an element of $G$ of order $rs$. 
Let $A_n$ ($S_n$) denote the alternating (symmetric) group of degree $n$. We prove that if $G$ is a finite group with $\Gamma(G)=\Gamma(A_n)$ or $\Gamma(G)=\Gamma(S_n)$, where $n\geq19$, then there exists a normal subgroup $K$
of $G$ and an integer $t$ such that $A_t\leq G/K\leq S_t$ and $|K|$ is divisible by at most one prime greater than $n/2$.
\end{abstract}

\section{Introduction}

Let $G$ be a finite group and $\omega(G)$ be its spectrum, that is, the set of orders of its elements. The {\it prime graph} $\Gamma(G)$ of $G$ is defined as follows. The vertex set is the set $\pi(G)$ of all prime divisors of the order of $G$. Two distinct primes $r, s\in\pi(G)$ regarded as vertices of the graph are adjacent 
if and only if $rs\in\omega(G)$. The alternating and symmetric groups of degree $n$ are denoted by $A_n$ and $S_n$, respectively. 

It is proved in \cite{13Gor} that if $L=A_n$, where $n\geq5$ with $n\neq6,10$, and $G$ is a finite group such with $\omega(G)=\omega(L)$ then $G\simeq L$. If $L=S_n$, then the same conclusion is established for $n\neq2,3,4,5,6,8,10$, see \cite{14Gor} and \cite{16GorGrish}. Therefore, for almost all $n$ the groups $A_n$ and $S_n$ are characterized by spectrum in the class of finite groups. Clearly, if $\omega(G)=\omega(H)$ then $\Gamma(G)=\Gamma(H)$. However, the converse implication is false in general, for instance, $\Gamma(A_5)=\Gamma(A_6)$ and $4\in\omega(A_6)\setminus\omega(A_5)$. As \cite{17Star} shows, if $A_n$ is characterized by its prime graph in the class of finite groups then $n=p$ or $n=p+1$, where $p$ is a prime such that $p-2$ is prime as well. 

Before those general results for $A_n$ and $S_n$ were proved, I.~Vakula \cite{10Vakula} showed that if $G$ is a finite group and $\omega(G)=\omega(A_n)$, where $n>21$, then $G$ has a chief factor isomorphic to $A_t$ for some $t$ with $\pi(A_t)=\pi(A_n)$.
The main goal of this paper is to prove that the same holds for $G$ provided that $\Gamma(G)=\Gamma(A_n)$ or $\Gamma(G)=\Gamma(S_n)$ for sufficiently large $n$.

\begin{theorem}
Suppose that $L\simeq A_n$ or $L\simeq S_n$, where $n\geq19$, and take the largest prime $p$ not greater than $n$. 
If $G$ is a finite group with $\Gamma(G)=\Gamma(L)$, then there exists a normal subgroup $K$ of $G$ such that $A_t\leq G/K\leq S_t$, where $t\geq p$. Moreover, the number $|K|$ is coprime to $p$ and divisible by at most one prime greater than $n/2$.
\end{theorem}

Observe that in this theorem $K$ can be an arbitrary group. Namely, it was noted in \cite{17Star} that
if $r\leq n-p$ for each $r\in\pi(K)$ then the prime graph of every extension of $A_n$ by $K$ coincides with $\Gamma(A_n)$. For example, $\Gamma({A}_5\times{A}_{28})=\Gamma(A_{28})=\Gamma({A}_5\wr{A}_{28})$. Nevertheless, if $n-p$ is bounded then $K$ has some restrictions on its composition structure: \cite{Vas05} shows that if $n-p\leq3$ then $K$ is solvable. As a consequence of our main theorem we prove that $K$ is solvable even if $n-p=4$.

\begin{corollary} Suppose that $n\geq19$ and consider a finite group $G$ with $\Gamma(G)=\Gamma(A_n)$ or $\Gamma(G)=\Gamma(S_n)$. If $p$ is the largest prime not greater than $n$ and $n-p\leq4$
then $A_t\leq G/K\leq S_t$, where $K$ is the solvable radical of $G$ and $t$ is an integer satisfying $p\leq t\leq p+4$. In particular, $G$ has a unique nonabelian composition factor.
\end{corollary}

\section{Preliminaries}

%\begin{lemma}\cite[Lemma~2.2]{ZavAlt16} Let $S=P_1\times...\times{P}_r$, where $P_i$ 
%are isomorphic nonabelian simple groups. Then $\Aut(S)\simeq(\Aut(P_1)\times...\times\Aut(P_r)).S_r$.
%\end{lemma}
\begin{lemma}{\rm\cite[Lemma~1]{17Star}}\label{l:semidir} Let $N$ be a normal elementary abelian subgroup of $G$ and $H=G/N$. Let $G_1=N\rtimes H$ be the natural semidirect product of $N$ by $H$. Then $\Gamma(G)=\Gamma(G_1)$.
\end{lemma}

\begin{lemma}{\rm\cite[Lemma~14]{10Vakula}}\label{l:Vakula} Let $S$ be a finite simple group.
Then each odd prime divisor of $Out(S)$ either divides $|S|$ or is less than or equal to $m/2$, where $m=max_{p\in\pi(S)}p$.
\end{lemma}

\begin{lemma}{\rm\cite[Lemma 3.6]{15Vas}}\label{l:hallhigman} For distinct primes $s$ and $r$, consider
a semidirect product $H$ of a normal $r$-subgroup $T$ and a cyclic 
subgroup $C=\langle g\rangle$ of order~$s$ with $[T,g]\neq1$.
Suppose that $H$ acts faithfully on a vector space $V$ of positive characteristic
$t$ not equal to~$r$. If the minimal polynomial of $g$ on $V$ does not equal $x^s-1$ then

\emph{(i)} $C_T(g)\neq1;$

\emph{(ii)} $T$ is nonabelian\emph{;}

\emph{(iii)} $r=2$ and $s$ is a Fermat prime.
\end{lemma}

For a positive integer $n$, we denote the set of prime divisors of $n$ by $\pi(n)$.

If $n$ is a nonzero integer and $r$ is an odd prime with $(r,n)=1$ then $e(r,n)$
denotes the multiplicative order of $n$ modulo $r$. Given an odd integer $n$, we put $e(2,n)=1$ if
$n\equiv 1\pmod 4$, and $e(2,n)=2$ otherwise.

Fix an integer $a$ with $|a|>1$. A prime $r$ is said to be a \emph{primitive prime divisor} of
$a^i-1$ whenever $e(r,a)=i$. Denote by $r_i(a)$ some primitive prime divisor of $a^i-1$, if exists, and by $R_i(a)$ 
the set of all these divisors. Zsigmondy proved \cite{Zs} that primitive prime divisors exist for almost all pairs $(a,i)$.

\begin{lemma}\emph{(Zsigmondy \cite{Zs})}\label{l:zsigmondy}
Given an integer $a$ with $|a|>1$, for every positive integer $i$ the set $R_i(a)$ is nonempty
except for the pairs $(a,i)\in\{(2,1),(2,6),(-2,2),(-2,3),(3,1),(-3,2)\}$.
\end{lemma}

For a classical group $S$, we denote by $\prk(S)$ its dimension if $S$ is a linear or unitary
group and its Lie rank if $S$ is a symplectic or orthogonal group.

\begin{lemma}{\rm\cite[Lemma 3.8]{15Vas}}\label{l:adjanisotrop}
For a simple classical group  $S$ over a field of order $u$ and characteristic $v$ with
$\prk(S)=m\geq4$, put

$$j=\begin{cases}
m & \text{ if }S\simeq L_{m}(u),\\
2m-2 & \text{ if either }S\simeq O^+_{2m}(u)\text{ or }S\simeq U_{m}(u)\text{ and }m\text{ is even,}\\
2m & \text{ otherwise.}
\end{cases}$$

Then $(r_j(u),|P|)=1$ for every proper parabolic
subgroup $P$ of~$S$. If $i\neq j$ and a primitive prime divisor $r_i(u)$ lies in $\pi(L)$ then
there is a proper parabolic subgroup $P$ of $S$ such that $r_i(u)$ lies in~$\omega(P)$. In
particular, if two distinct primes $r,s\in\pi(S)$ do not divide the order of any proper parabolic
subgroup of $S$ then $r$ and $s$ are adjacent in $\Gamma(S)$.
\end{lemma}

\begin{lemma}{\rm\cite[Lemma 3.7]{15Vas}}\label{l:good} 
Let $S$ be a simple classical group over a field of order~$u$ and characteristic $v$ with $\prk(S)>4$.
Suppose that a prime $s$ divides the order of a proper parabolic subgroup of $S$ and $(s,v(u^2-1))=1$. 
Then $S$ includes a subgroup $H$ which is a semidirect product of a normal $v$-subgroup $T$ and a cyclic subgroup
$C=\langle g\rangle$ of order $s$ with $[T,g]\neq1$ and at least one of three assertions in Lemma~\emph{\ref{l:hallhigman}} fails for~$H$.
\end{lemma}

Recall that a subset of vertices of a graph is called a {\it coclique} if
no pair of its vertices are adjacent. Denote by $t(G)$ the greatest size of a coclique in $\Gamma(G)$.
We refer to a coclique containing $r$ as an $\{r\}$-coclique.
If $r\in\pi(G)$ then $t(r,G)$ denotes the greatest size of $\{r\}$-cocliques. 

\begin{lemma}{\rm\cite[Lemma~15]{17Star}}\label{l:r1r2r3r4}
Consider a simple classical group $S$ and put $m=\prk(S)$. If $t(S)\geq15$ then
there exist $r_1,r_2,r_3,r_4\in\pi(S)\setminus\{v\}$ such that the following claims hold:

\noindent$(i)$ $r_1$ is adjacent to $r_2$ and not adjacent to $r_4$ in $\Gamma(S)$,

\noindent$(ii)$ $r_3$ is adjacent to $r_4$ and not adjacent to $r_2$ in $\Gamma(S)$,

\noindent$(iii)$ if $S$ is unitary then $e(r_i,-u)\geq\max(7,(m-5)/3)$ for $1\leq i\leq 4$ and
otherwise $e(r_i,u)\geq\max(7,(m-5)/3)$ for $1\leq i\leq 4$.
\end{lemma}

\begin{lemma}\label{l:R(-u)} 
Consider integers $a$ and $i$ with $|a|>1$ and $i>0$. The the following claims hold:

\noindent$(i)$ if $i$ is odd then $R_i(-a)=R_{2i}(a)$, and if $i$ is a multiple of $4$ then $R_i(-a)=R_i(a)$;

\noindent$(ii)$ if $r\in R_i(a)$ and $r>2$ then $r-1$ is divisible by $i$.
\end{lemma}
\begin{proof} The first assertion is a consequence of \cite[Lemma 1.3]{15Vas}. Now we prove $(ii)$.
If $a>0$ then by Fermat little theorem we know that $a^{r-1}-1$ is divisible by $r$.
Since $i=e(r,a)$, we conclude that $i$ divides $r-1$. Let $a<0$. If $i$ is odd or $i$ is divisible by 4 then $(i)$ implies that either $e(r,-a)=i$ or $e(r,-a)=2i$ and hence $r-1$ is divisible by $i$. It remains to consider the case $i=4k+2$ with an integer $k$. Then $(i)$ implies that $e(r,-a)=i/2$ and hence $r-1$ is divisible by $2k+1$.
However, we know that $r-1$ is even. Therefore, $4k+2$ divides $r-1$, as required.
\end{proof}

\begin{lemma}{\rm\cite[Lemma~1.1]{Vas05}}\label{3-coclique}
Consider a finite group $G$ with a normal series of subgroups $1\leq K\leq M\leq G$. Then three primes
$p$, $q$ and $r$ such that $p$ divides $|K|$, $q$ divides $|M/K|$, and $r$ divides $|G/M|$ 
cannot be pairwise nonadjacent in $\Gamma(G)$.
\end{lemma}

\begin{lemma}\label{l:S-factor} For a finite group $G$ take a coclique $\rho$ in $\Gamma(G)$ with $|\rho|\geq3$.
Then the following claims hold:

$(i)$ there exists a nonabelian composition factor $S$ of $G$ and a normal subgroup $K$ of $G$ such that $S\leq\overline{G}=G/K\leq Aut(S)$ 
and $|\pi(S)\cap\rho|\geq2$. 

$(ii)$ If $\rho'$ is a coclique in $\Gamma(G)$ with $|\rho'|\geq3$ and $|\pi(S)\cap\rho'|\geq1$ then $|G|/|S|$ is divisible by at most one element of $\rho'$. In particular, $|\pi(S)\cap\rho|\geq|\rho'|-1$ and $S$ is a unique composition factor of $G$ with $|\pi(S)\cap\rho'|\geq2$. 

\end{lemma}
\begin{proof} Consider a chief series $1=N_0\leq N_1\leq...\leq N_t=G$ for $G$. Lemma~\ref{3-coclique} implies that for some $i\geq1$ at least two elements of $\rho$ divide $|N_i:N_{i-1}|$. Put $\overline{G}=G/N_{i-1}$ and $\overline{N_i}=N_i/N_{i-1}$. Since $\overline{N_i}$ is a direct product of isomorphic simple groups and $t(\overline{N_i})\geq2$, it follows that $\overline{N_i}$ is isomorphic to a nonabelian simple group $S$. Since $C_{\overline{G}}(\overline{N_i})\triangleleft\overline{G}$ and $C_{\overline{G}}(\overline{N_i})\cap\overline{N_i}=1$, we may assume that there exists a normal subgroup $K$ of $G$ such that $S$ is normal in $G/K$ and $C_{G/K}(S)=1$. Then $G/K$ is isomorphic to a subgroup of $\Aut(S)$. Therefore, $S$ and $K$ are the required groups for assertion $(i)$.

Take a coclique $\rho'$ in $\Gamma(G)$ with $|\rho'|\geq3$ and $|\pi(S)\cap\rho'|\geq1$. Suppose that $|G|/|S|$ is divisible by two distinct primes lying in $\rho'$. 
We claim that there exist distinct primes $t,r,s\in\rho'$ such that $t\in\pi(S)$ and $r,s\in\pi(|G|/|S|)$.
We have $|\pi(S)\cap\rho'|\geq1$. Assume that $|\pi(S)\cap\rho'|\geq3$. Let $r$, $s$ be distinct primes from $\pi(|G|/|S|)\cap\rho'$ and $t\in\pi(S)\cap\rho'$ be a prime different from $r$ and $s$. 
If $|\pi(S)\cap\rho'|=2$ then there exist $r\in\rho'\setminus\pi(S)$ and $s\in\pi(|G|/|S|)$ such that $s\neq r$. Since both primes of $\pi(S)\cap\rho'$ are different from $r$, 
there exists $t\in\pi(S)\cap\rho'$ with $t\neq r$ and $t\neq s$. Finally, assume that $|\pi(S)\cap\rho'|=1$. 
Then take $t\in\pi(S)\cap\rho'$ and distinct primes $r$, $s$ in $\rho'\setminus\pi(S)$. In all cases, we find the required numbers. Since $\Out(S)$ is solvable, Lemma~\ref{3-coclique} implies that $|\overline{G}|/|S|$ is divisible by at most one number in $\{r, s\}$. If $|\overline{G}|/|S|$ is divisible by $r$ or $s$ then the other prime divides $|K|$, again we arrive at a contradiction with Lemma~\ref{3-coclique}. Therefore, $r,s\in\pi(K)$. 
Consider the socle series $1=T_0\triangleleft T_1...\triangleleft T_m=K$ for $K$, where $T_{j}/T_{j-1}=\operatorname{Soc}(K/T_j)$ for $j\geq1$. By Lemma~\ref{3-coclique}, the primes $r$ and $s$ divide $|T_{j}/T_{j-1}|$ for some $j$. Since $T_j/T_{j-1}$ is a direct product of simple groups, there exists a nonabelian simple subgroup $L$ in $T_{j}/T_{j-1}$ with $r,s\in\pi(L)$. Clearly,  $L$ is a characteristic subgroup in $T_{j}/T_{j-1}$, so $L$ is normal in $G/T_{j-1}$. Put $\widetilde{G}=G/T_{j-1}$. Then $N_{\widetilde{G}}(L)=\widetilde{G}$. Since $S$ is a composition factor of $G/K$, it follows that $S$ is a composition factor of $N_{\widetilde{G}}(L)/L$. The normal series $1\leq C_{\widetilde{G}}(L)\leq N_{\widetilde{G}}(L)=\widetilde{G}$ can be refined to a composition series. Since $N_{\widetilde{G}}(L)/C_{\widetilde{G}}(L)$ is isomorphic to a subgroup of $\Aut(L)$, we infer that the composition factors of $N_{\widetilde{G}}(L)/C_{\widetilde{G}}(L)$ are abelian, in particular, not isomorphic to $S$. By the Jordan-H\"{o}lder theorem, $S$ is a composition factor of $C_{\widetilde{G}}(L)$, whence $t$ and $r$ are adjacent in $\Gamma(G)$; a contradiction. 
So $|G|/|S|$ is divisible by at most one element of $\rho'$ and hence $S$ is a unique composition factor of $G$ with $|\pi(S)\cap\rho|\geq2$. The lemma is proved.
\end{proof}

\begin{lemma}\label{l:primes_in_K}
Consider a finite group $G$ and a normal subgroup $K$ of $G$. If $r$ divides the order of a nonabelian composition factor $S$ of $K$,
then $r$ is adjacent to every $s\in\pi(G/K)\setminus\pi(\Aut S)$ in $\Gamma(G)$.
\end{lemma}
\begin{proof} Assume the contrary and take $s\in\pi(G/K)\setminus\pi(\Aut S)$ with $rs\not\in\omega(G)$.
Consider the socle series $1=N_0<N_1...<N_k=K$ for $K$, so $N_i/N_{i-1}=\operatorname{Soc}(K/N_{i-1})$ for $i>0$.
We may assume that $S\triangleleft N_1$. Take an order-$s$-element $g$ of $N_G(N_1)$. Then $g$ acts by conjugation on the set $\{S^h~|~h\in G\}$. Then the length of the orbit of $g$ equals either 1 or $s$. If $g$ normalizes $S$ then $g$ centralizes $S$ because $s$ does not divide $|\Aut S|$; a contradiction with $rs\not\in\omega(G)$. So the length equals $s$. Then $g$ normalizes a group isomorphic to $S^s=S\times S...\times S$. If $h\in S$ is an element of order $r$ then $g$ centralizes $h\cdot h^g\cdot...\cdot h^{g^{s-1}}$ and hence $rs\in\omega(G)$; a contradiction.
\end{proof}

For two real numbers $a<b$, denote by $\pi(a,b)$ the number of primes $s$ satisfying $a<s\leq b$.

\begin{lemma}{\rm\cite[Corollary~3]{piFunc}}\label{l:pix} 
If $x\geq20\frac{1}{2}$ then $\pi(x,2x)>3x/(5\ln x)$ and if $x>1$ then $\pi(x,2x)<7x/(5\ln x)$.
\end{lemma}

\begin{lemma}{\rm\cite[Proposition~5.4]{Dusart}}\label{l:Dusart}
For all $x\geq89693$, there exists a prime $p$ such that $x<p\leq x(1+\frac{1}{ln^3x})$.
\end{lemma}

\begin{lemma}\label{l:t(S)_restriction}
For every integer $n\geq167$, we have $t(S_n)\geq t(A_n)\geq 16$.
\end{lemma}
\begin{proof} Put $x=\lfloor n/2\rfloor$. If $r$ and $s$ are distinct primes with $r\geq x$ and $s\geq x$ then
$r+s>n$. Therefore, $t(S_n)\geq t(A_n)\geq~\pi(x,2x)$. Suppose that $n\geq400$. 
Then $x\geq200$. Lemma~\ref{l:pix} implies that $t(A_n)>16$. If $167\leq p<400$ then the inequality $t(A_n)\geq 16$ can be verified using a table of primes, see \cite{PrimeTable} for instance.
\end{proof}

\begin{lemma}\label{l:r in S} Suppose that $n\geq17$. Take the largest prime $p$ less than or equal to $n$ and the largest prime $s$ less than $p$. Then there exists a prime $r\leq s/2$ such that $p+r>n$.
\end{lemma}
\begin{proof} Lemma~1.9 of \cite{15Vas} shows that if $m\geq30$ then the interval $(5m/6,m)$ contains a prime. Assume that $n\geq87$. Consider the primes $p$ and $s$ from the hypothesis.
Observe that $p>5n/6\geq72$, hence $s>5p/6\geq60$. Put $s'=\lfloor s/2\rfloor$ and observe that there exists a prime $r$ in $(5s'/6,s')$. We claim that $r$ is the required prime. Since $s'\geq(s-1)/2$ and $s\geq25n/36$, we infer that $s'\geq(25n/36-1)/2=25n/72-1/2$. Therefore, $r>(5/6)(25n/72-1/2)=125n/432-5/12$ and hence $p+r>5n/6+125n/432-5/12=485n/432-5/12$. Since $n\geq87$, it follows that $485n/432-5/12>n+1-5/12>n$. Thus, $p+r>n$. 
If $17\leq n\leq 87$ then the existence of $r$ can be verified directly.
\end{proof}

\begin{lemma}\label{l:S is alt} Consider a finite group $G$ satisfying $\Gamma(G)=\Gamma(A_n)$ or $\Gamma(G)=\Gamma(S_n)$, where $n\geq17$.
Suppose that $A_t\leq G/K\leq S_t$ for a normal subgroup $K$ of $G$ and an integer $t$ such that $\pi(A_t)$ contains at least two primes greater than $n/2$. Then $\pi(A_t)=\pi(A_n)$
and $|K|$ is coprime to $p$ and divisible by at most one prime greater than $n/2$.
\end{lemma}
\begin{proof} Put $\Omega=\{p\in\mathbb{P}~|~n/2<p\leq n\}$. Applying Lemma~\ref{l:pix}, we see that $|\Omega|\geq3$. Clearly, the set $\Omega$ is a coclique in $\Gamma(G)$. 
Denote the largest prime less than or equal to $n$ by $p$, and the largest prime less than $p$ by $s$.
Lemma~\ref{l:S-factor} yields $t\geq s$ and $|K|$ is divisible by at most one prime greater than $n/2$. It suffices to prove that $t\geq p$ and $p\not\in\pi(K)$.
Assume on the contrary that $t<p$. Then $p\in\pi(K)$. Take a Sylow $p$-subgroup $P$ of $K$.
The Frattini argument implies that $N_G(P)/N_K(P)\simeq G/K$. By Lemma~\ref{l:r in S}, there exists a prime $r\in\pi(A_t)$ such that $r\leq s/2$ and $p+r>n$. Since $r\leq t/2$, we conclude that $A_t$ contains an elementary abelian subgroup of order $r^2$, in particular, the Sylow $r$-subgroups of $A_t$ are noncyclic. Therefore, the Sylow $r$-subgroups of $N_G(P)$ are also noncyclic, whence $pr\in\omega(N_G(P))\subseteq\omega(G)$. Since $p+r>n$, we arrive at a contradiction.
This reasoning also shows that $p\not\in\pi(K)$.
\end{proof}

\section{Proof of the Theorem: the case $n\geq1000$}
Suppose that $G$ is a finite group such that $\Gamma(G)=\Gamma(L)$, where $L\simeq A_n$ or $L\simeq S_n$ with $n\geq1000$. Put $\Omega_L=\{p\in\mathbb{P}~|~n/2\leq p\leq n\}$. By Lemma~\ref{l:S-factor}, there exists a normal subgroup $K$ of $G$ such that $S\leq\overline{G}=G/K\leq Aut(S)$ for a nonabelian simple group $S$ and $|\pi(S)\cap\Omega_L|\geq|\Omega_L|-1$. 
Then $\Omega_S=\pi(S)\cap\Omega$ is a coclique in $\Gamma(S)$.
Lemma~\ref{l:pix} shows that $t(L)\geq|\Omega_L|\geq\pi(500,1000)>48$.
Therefore, $|\Omega_S|\geq|\Omega_L|-1\geq48$. According to Tables~1,~4 of \cite{11VasVd.t}, we have $t(S)\leq12$ when $S$ is a sporadic simple group or a simple exceptional group of Lie type. Therefore, $S$ is an alternating group or a simple classical group. 
If $S$ is an alternating group, then the theorem holds by Lemma~\ref{l:S is alt}. It suffices to prove that $S$ cannot be isomorphic to a simple classical group. 

\begin{pr} $S$ is not a simple classical group. 
\end{pr}
\begin{proof}
Assume on the contrary that $S$ is a simple classical group of Lie type over a field of order $u$. Let $m=\prk(S)$ and $u=v^k$, where $v$ is a prime and $k$ is a positive integer. Put $\varepsilon=-$ if $S$ is a unitary group and $\varepsilon=+$ otherwise.
\begin{lemma}\label{l:r in OutS} 
If $i\geq\max((m-5)/3,7)$ and $R_i(\varepsilon{u})\subset\pi(S)$ then there exists $r\in R_i(\varepsilon{u})$ such that $r$ does not divide $|\Out{S}|$.
\end{lemma}
\begin{proof} Since $i\geq7$, we have $R_i(\varepsilon{u})\neq\varnothing$. Let $\varepsilon=-$. Lemma~\ref{l:R(-u)} yields $R_i(-u)=R_{i'}(u)$, where $i'\in\{i/2,i,2i\}$. If $\varepsilon=+$ we put $i'=i$. Since $i\geq7$, we obtain $i'\geq4$. If $k\neq1$ then the claim is clear. If $k\geq2$ then $i'k\geq8$ and there exists $r\in R_{i'k}(v)$. Therefore, $r$ divides $u^{i'}-1$ and does not divide $u^j-1$ for $j<i'$. Hence, $r\in R_{i'}(u)$. By Lemma~\ref{l:R(-u)}, 
$r-1$ is divisible by $i'k$. In particular, $r>3$ and $(r,k)=1$. Whence $(r, |\Out{S}|)=1$ and so $r$ is the required prime.
\end{proof}

\begin{lemma}\label{l:r in K}
%Если $i\geq\max((m-5)/3,7)$, то любое $r\in R_i(\varepsilon{u})$ не делит $|K|$.
Let $\{r,p_1,p_2\}$ be a coclique in $\Gamma(G)$ such that $p_1$ and $p_2$ are elements of $\pi(S)$ which are not Fermat primes and $(p_1p_2, v(u^2-1))=1$. Then $|K|$ is not divisible by $r$.
\end{lemma}
\begin{proof} 
Assume on the contrary that $r$ divides $|K|$. Lemma~\ref{l:S-factor} yields $p_1,p_2\in\pi(S)$. Lemma~\ref{l:adjanisotrop} implies that at least one number of $p_1$ and $p_2$ divides the order of some parabolic subgroup $P$ of $S$. Let $p_1$ be this number.
To get a contradiction, we verify that $p_1r\in\omega(G)$. Take a Sylow $r$-subgroup $R$ of $K$. Applying Frattini's argument, we may assume that $R$ is a normal elementary abelian subgroup in $G$. Then $C_G(R)$ is a normal subgroup of $G$ and, since $S$ is simple and $rp_1\not\in\omega(G)$, we infer that $C_G(R)\leq K$. By the Schur--Zassenhaus theorem, $C_G(R)$ is a direct product of $R$ and some $r'$-subgroup. Therefore we may assume that $C_G(R)\leq R$. By Lemma~\ref{l:good}, there exists a subgroup $H=T\rtimes C$ of $S$, where $T$ is a $v$-group and $C$ is a cyclic group of order $p_1$ such that $[T,C]\neq1$. Denote by $G_1$ and $K_1$ the preimages of $H$ and $T$ in $G$.  
Now we prove that there exists a group with the properties of $H$ in $G_1$. Since $p_1$ does not divide the order of $K$, there exists a subgroup $C_1$ of order $p_1$ in $G_1$ whose image in $G/K$ is $C$. Then $G_1=K_1\rtimes C_1$.
Denote by $T_1$ a Sylow $v$-subgroup in $K_1$ whose image in $G_1/K$ is $T$. 
Frattini's argument implies that $N_{G_1}(T_1)K_1=G_1$, and so $N_{G_1}(T_1)$ includes a subgroup $C_2$ of order $p_1$. 
We claim that $H_1=T_1\rtimes C_2$ is the required group. Suppose that $[T_1,C_2]=1$. Then the same must hold for the images of $T_1$ and $C_2$ in $G_1/K$, hence $[T,C]=1$; a contradiction. 
Since $C_G(R)\leq R$, it follows that $H_1$ acts faithfully on $R$ by conjugation.
Since $p_1$ is not a Fermat prime, the group $H_1$ does not satisfy to $(iii)$ of Lemma~\ref{l:hallhigman}. Hence, an element of order $p_1$ in $H_1$ centralizes
some non-identity element of $R$, so $p_1r\in\omega(G)$; a contradiction. 
\end{proof}

\begin{lemma}\label{l:ri in K} If $i\geq\max((m-5)/3,7)$ and $r\in R_i(\varepsilon{u})$ then $r$ does not divide $|K|$.

\end{lemma} 
\begin{proof}
Assume on the contrary that $r$ divides $|K|$. Denote the $j$-th prime number by $p_j$ and the largest prime less than or equal to $n$ by $p_y$. Firstly, we suppose that $r<p_{y-2}$
and $r$ is nonadjacent to $p_{y-2}$ in $\Gamma(G)$. Then $\{r, p_{y-2}, p_{y-1}, p_y\}$ is a coclique in $\Gamma(G)$.  Lemma~\ref{l:S-factor} shows that $p_{y-2}$, $p_{y-1}$, and $p_y$ divide the order of $S$
and are coprime to $|G|/|S|$. Since $n/2<p_{y-2}<p_{y-1}<p_{y}\leq n$, at most one number among $p_{y-2}$, $p_{y-1}$, and $p_y$ is a Fermat prime. Observe that each of these three primes does not divide $v(u^2-1)$. Indeed, if $p\in\{p_{y-2},p_{y-1},p_y\}$ then $t(p,S)\geq |\Omega_S|\geq48$; however, $t(v,S)\leq4$ according to 
\cite[Tables 4, 6]{05VasVd.t}. Moreover, $t(r_1(u),S),t(r_2(u),S)\leq4$ by the adjacency criteria for the prime graphs of simple groups of Lie type \cite{05VasVd.t}.
Considering two primes from $\{p_{y-2}, p_{y-1}, p_y\}$ that are not Fermat primes, we arrive at a contradiction by Lemma~\ref{l:r in K}.
Thus, either $r\in\{p_{y-2}, p_{y-1}, p_y\}$ or $r$ is adjacent to $p_{y-2}$. In the first case, replacing $r$ by $p_{y-3}$ in the set $\{p_{y-2}, p_{y-1}, p_y\}$, we arrive at a contradiction as above, so $r+p_{y-2}\leq n$.

According to \cite[Table~1]{15Vas}, we have $(3m+5)/4\geq t(S)$. Since $t(G)\geq\pi(n/2,n)$, Lemma~\ref{l:pix} implies that $t(G)\geq3n/(10\ln(n/2))$. Lemma~\ref{l:S-factor} yields $t(S)\geq~t(G)-1$, whence
\begin{equation}\label{eq:1}
(3m+5)/4\geq 3n/(10\ln(n/2))-1,
\end{equation}
\begin{equation}\label{eq:2}
n-p_{y-2}\geq r\geq1+(m-5)/3=(m-2)/3.
\end{equation}

Observe that
\begin{equation}\label{eq:3}
\frac{m-2}{3}=\frac{3m-6}{9}=\frac{4}{9}\cdot\frac{3m-6}{4}=\frac{4}{9}\cdot\frac{3m+5-11}{4}.
\end{equation}

Combining (\ref{eq:2}) with (\ref{eq:3}) and (\ref{eq:1}), we obtain 
$$n-p_{y-2}\geq\frac{4}{9}\cdot\frac{3n}{10\ln(n/2)}-\frac{4}{9}-\frac{11}{9}=\frac{2n}{15\ln(n/2)}-\frac{15}{9}>\frac{2n}{15\ln(n/2)}-2.$$
Hence, $n-p_{y-2}>\frac{2n}{15\ln(n)}-2$.

Suppose that $n>90000$. Since $89983$ and $89989$ are primes, we have $p_{y-1}\geq 89983$. Lemma~\ref{l:Dusart} implies that $n<p_{y+1}\leq p_y(1+\frac{1}{\ln^3(p_y)})$ and hence $n-p_y<\frac{p_y}{\ln^3(p_y)}\leq\frac{n}{\ln^3(n)}$.
Similarly we infer that $p_y-p_{y-1}\leq\frac{p_{y-1}}{\ln^3(p_{y-1})}<\frac{n}{\ln^3(n)}$ and $p_{y-1}-p_{y-2}<\frac{n}{\ln^3(n)}$.
So $n-p_{y-1}=(n-p_y)+(p_y-p_{y-1})+(p_{y-1}-p_{y-2})<\frac{3n}{\ln^3(n)}$. Therefore, $\frac{3n}{\ln^3(n)}>\frac{2n}{15\ln(n)}-2$.
Since $n>90000$, we have $\ln^2(n)>30$ and hence $\frac{3n}{30\ln(n)}>\frac{3n}{\ln^3(n)}>\frac{2n}{15\ln(n)}-2$. Therefore, $60>\frac{n}{\ln(n)}$.
This inequality is false if $n>90000$; a contradiction.

Suppose that $15000<n\leq90000$. Then according to \cite[Table~1]{PrimeGap} we know that $n-p_y$, $p_y-p_{y-1}$, $p_{y-1}-p_{y-2}\leq72$. Hence, $n-p_{y-2}\leq216$. We found above that
$n-p_{y-2}>\frac{2n}{15\ln(n/2)}-2$. Therefore, $216>\frac{30000}{15\ln(7500)}-2>222$; a contradiction.

Suppose that $1000\leq n\leq 15000$. In this case there are no Fermat primes between $n/2$ and $n$. Arguing as above, we see that $r+p_{y-1}\leq n$. Hence,
$n-p_{y-1}\geq r\geq1+(m-5)/3=(m-2)/3$. Therefore, $m\leq3(n-p_{y-1})+2$. On the other hand, we know that $(3m+5)/4\geq t(S)\geq t(G)-1\geq\pi(n/2,n)-1$, whence
$m\geq(4\pi(n/2,n)-9)/3$. Thus, $3(n-p_{y-1})+2>(4\pi(n/2,n)-9)/3$.
Divide the segment $[1000,15000]$ of values of $n$ into three segments $[1000,1300]$, $[1300,2000]$, $[2000,15000]$.
By performing straightforward calculations, it is possible to compute on these segments the minimal value of $3(n-p_{y-1})+2$ and the maximal value of $(4\pi(n/2,n)-9)/3$. Table~\ref{tab:tL} shows the results.

\begin{table}[!th]
\caption{The maximal values of $3(n-p_{y-1})+2$ and the minimal values of $(4\pi(n/2,n)-9)/3$}\label{tab:tL}
\begin{center}
{\begin{tabular}{|c|c|c|c|}
  \hline
   $[L,R]$ & $\max\limits_{n\in[L,R]}3(n-p_{y-1})+2$  & $\min\limits_{n\in[L,R]}(4\pi(n/2,n)-9)/3$  \\
  \hline
  $[1000, 1300]$  &  83 $(n=1150, p_y=1129, p_{y-1}=1123)$ & 93 $(n=1006)$  \\ 
  $[1300, 2000]$  & 119 $(n=1360, p_y=1327, p_{y-1}=1321)$ & 121 $(n=1300)$ \\
  $[2000, 15000]$ & 149 $(n=9600, p_y=9587, p_{y-1}=9551)$  & 177 $(n=2000)$  \\
  \hline
  \end{tabular}
}
\end{center}
\end{table}
In all cases $(4\pi(n/2,n)-9)/3<3(n-p_{y-1})+2$; a contradiction. This completes the proof of Lemma~\ref{l:ri in K}.
\end{proof}
Now we are ready to obtain a contradiction justifying the proposition.
By Lemma~\ref{l:r1r2r3r4}, there exist primes $r_1,r_2,r_3,r_4\in\pi(S)\setminus\{v\}$ such that $e(r_i,\varepsilon{u})\geq\max(7,(m-5)/3)$ 
for $1\leq i\leq4$, with $r_1r_2\in\omega(S)$, $r_3r_4\in\omega(S)$, $r_1r_4\not\in\omega(S)$, and $r_2r_3\not\in\omega(S)$. Note that, by \cite{05VasVd.t}, the existence of an edge between primes $r$ ans $s$ in $\Gamma(S)$ depends only on numbers $e(r,u)$ and $e(s,u)$. By Lemmas~\ref{l:r in OutS} and \ref{l:ri in K}, we can choose $r_1$, $r_2$, $r_3$, and $r_4$ such that
$(r_i,|\overline{G}/S|\cdot|K|)=1$ for $1\leq i\leq4$. Therefore, $r_1r_2\in\omega(L)$, $r_3r_4\in\omega(L)$, $r_1r_4\not\in\omega(L)$, and $r_2r_3\not\in\omega(L)$.
Hence, $r_1+r_2\leq n$, $r_1+r_4>n$, $r_3+r_4\leq n$, and $r_3+r_2>n$. The first and second inequalities yield $r_4>r_2$, 
and the third and fourth inequalities yield $r_4<r_2$; a contradiction.
\end{proof}

\section{Proof of the Theorem: the case $19\leq n<1000$}
Suppose that $G$ is a finite group such that $\Gamma(G)=\Gamma(L)$, where $L\simeq A_n$ or $L\simeq S_n$ for an integer $n$ satisfying $19\leq n\leq 1000$. Put $\Omega_L=\{p\in\mathbb{P}~|~n/2\leq p\leq n\}$. Lemma~\ref{l:pix} yields $|\Omega_L|\geq3$. By Lemma~\ref{l:S-factor}, there exists a normal subgroup $K$ of $G$ such that $S\leq G/K\leq Aut(S)$ for a nonabelian simple group $S$. Moreover, if $\Omega_S=\pi(S)\cap\Omega$ then $\Omega_S$ is a coclique in $\Gamma(S)$ and $|\Omega_S|\geq|\Omega_L|-1$. Following \cite{Zavar}, we denote by $\mathfrak{S}_r$ the set of nonabelian simple groups $S$ such that $r\in\pi(S)$ and no element of $\pi(S)$ exceeds $r$. 
Denote the $i$-th prime number by $p_i$ and the index of the largest prime less than or equal to $n$ by $y$. 
Lemma~\ref{l:S-factor} implies that either $S\in\mathfrak{S}_{p_y}$ or $S\in\mathfrak{S}_{p_{y-1}}$.
If $S$ is the alternating group $A_t$ then $t\geq p_y$ by Lemma~\ref{l:S is alt} and the theorem is justified in this case. The following lemma completes the proof of the theorem.

\begin{lemma}\label{l:S is alt2} $S$ is an alternating group.
\end{lemma}
\begin{proof}

Assume the contrary. 

Suppose that $n\geq167$. Applying Lemmas~\ref{l:S-factor} and \ref{l:t(S)_restriction}, we obtain $t(S)\geq t(G)-1\geq15$. 
According to \cite{Zavar}, either $S\simeq L_2(p)$ or $S$ is a group in Table~3 of \cite{Zavar}, corresponding to $p_y$ or $p_{y-1}$. If $S=L_2(p)$ then $t(S)=3$ according to Table~2 of \cite{11VasVd.t}. Therefore, $S$ is a group  in Table~3 of \cite{Zavar}, in particular, $S$ is not a sporadic simple group. 
If $S$ is a simple classical group then $\prk(S)\leq12$. 
According to Tables~2, 3 of \cite{11VasVd.t}, we find that $t(S)\leq9$. If $S$ is an exceptional 
group of Lie type then $t(S)\leq12$ according to Table~4 of \cite{11VasVd.t}.
In all cases, we obtain $t(S)\leq12$; a contradiction.

Suppose that $100<n<167$. Then, as above, either $S\simeq L_2(p)$ or $S$ is a group in Table~3 of \cite{Zavar}, corresponding to $p_y$ or $p_{y-1}$. Now we show that $t(S)\leq8$.
If $S=L_2(p)$ then $t(L)=3$ according to Table~2 of \cite{11VasVd.t}. Assume that $S$ is a group in Table~3 of \cite{Zavar}, in particular, $S$ is not a sporadic group.
If $S$ is a linear or unitary group then $\prk(S)\leq12$, and hence $t(S)\leq6$ by \cite[Table~2]{11VasVd.t}. 
If $S$ is a symplectic or orthogonal group then $\prk(S)\leq8$, and hence $t(S)\leq7$ by \cite[Table~3]{11VasVd.t}. 
If $S$ is an exceptional group of Lie type then $S$ is not of type $E_8$ according to \cite[Table~3]{Zavar}.
Therefore, Table~4 of \cite{11VasVd.t} shows that $t(S)\leq8$. Thus, $t(S)\leq8$ in all cases. Then~$t(G)\leq t(S)+1\leq9$. However, it is easy to verify that $t(S_n)\geq t(A_n)\geq11$ if $n\geq101$; a contradiction.

Suppose that $19\leq n\leq100$. It is easy to verify that $p_{y-3}+p_{y-2}>n$, so $p_{y-3}$, $p_{y-2}$, $p_{y-1}$, and $p_y$ are pairwise nonadjacent in $\Gamma(G)$. Lemma~\ref{l:S-factor} implies that $|\{p_{y-3},p_{y-2},p_{y-1}\}\cap\pi(S)|\geq2$. Moreover, if $p_y\not\in\pi(S)$ then
$p_{y-3},p_{y-2},p_{y-1}\in\pi(S)$. Verifying these conditions for groups in $\mathfrak{S}_{p_y}$ using \cite[Table~1]{Zavar}, we infer that $p_y\in\{19,23,29,43\}$ and $S\in\{Fi_{23},Fi^\prime_{24},U_3(37), J_4, {^2}E_6(2)\}$. Now we consider each of these possibilities for $p_y$ and $S$. 

If $p_y=43$ then $S\in\{U_3(37),J_4\}$. Observe that $17,41\not\in\pi(S)$; therefore 17 and 41 divide $|G|/|S|$; a contradiction with Lemma~\ref{l:S-factor}, since $\{17,41,43\}$ is a coclique in $\Gamma(L)$.

If $p_y=29$ then $S=Fi'_{24}$ and $n\in\{29,30\}$. Since $19\not\in\pi(S)$, it follows that $19\in\pi(K)$. Assume that $n=29$.
Then $11$ and $17$ are adjacent in $\Gamma(G)$ and nonadjacent in $\Gamma(S)$. Therefore, either $11\in\pi(K)$ or $17\in\pi(K)$. 
Since $29$ is nonadjacent to any vertex in $\Gamma(G)$, an element of order 29 in $S$ acts fixed-point-freely on $K$, 
and so $K$ is nilpotent by the Thomson theorem. Therefore, either $11\cdot19\in\omega(G)$ or $17\cdot19\in\omega(G)$; a contradiction.
Assume that $n=30$. Then $13$ and $17$ are adjacent in $\Gamma(G)$ and nonadjacent in $\Gamma(S)$. Arguing as above, we find that either $13\cdot19\in\omega(G)$ or $17\cdot19\in\omega(G)$; a contradiction. 

If $p_y=23$ then $S=Fi_{23}$ or $S={^2}E_6(2)$. In the first case $19\in\pi(K)$.
If $n=23$ then $K$ is nilpotent as above. Since $55\not\in\omega(S)$, either 5 or 11 divides the order of $K$, hence either $5\cdot19\in\omega(G)$ or $11\cdot19\in\omega(G)$; a contradiction. If $n\geq24$ then 11 and 13 are adjacent in $\Gamma(G)$ and nonadjacent in $\Gamma(S)$, hence either 11 or 13 divides $|K|$. Since $\{11,19,23\}$ and $\{13,19,23\}$ are cocliques in $\Gamma(L)$, we arrive at a contradiction with Lemma~\ref{l:S-factor}. Suppose now that $S={^2}E_6(2)$. Then $23\in\omega(K)$. If $n=23$ then $5\cdot17\in\omega(G)\setminus\omega(S)$, so either 5 or 17 divides $|K|$. This contradicts Lemma~\ref{l:S-factor} because $\{5,19,23\}$ and $\{17,19,23\}$ are cocliques in $\Gamma(L)$.
If $n\geq24$ then we arrive at a contradiction as above because $11\cdot13\in\omega(G)\setminus\omega(S)$.

If $p_y=19$ then $S={^2}E_6(2)$ and $19\leq n\leq 22$. Observe that $7\cdot11\not\in\omega(S)$,
so either $7\in\pi(K)$ or $11\in\pi(K)$. Since $\{7,17,19\}$ and $\{11,17,19\}$ are cocliques in $\Gamma(G)$,
Lemma~\ref{l:S-factor} implies that $17,19\not\in\pi(K)$. Now we prove that $K$ is solvable.
Assume that $K$ is nonsolvable and $R$ is a nonabelian composition factor of $K$. All elements of $\pi(R)$ are at most
13, so $19\not\in\pi(\Aut R)$ by Lemma~\ref{l:Vakula}. Clearly, there exists a prime $s\in\pi(R)$
with $s>3$. Lemma~\ref{l:primes_in_K} implies that $19s\in\omega(G)$; a contradiction. Hence $K$ is solvable.
Suppose that $r\in\{7,11\}\cap\pi(K)$.
To obtain a contradiction, we show that either $17r\in\omega(G)$ or $19r\in\omega(G)$. Take a Hall $\{2,3\}'$-subgroup $P$ of $K$. Applying Frattini's argument, we may assume that $P$ is normal in $G$. Since $19\in\pi(N_G(P))$ and 19 is not adjacent to every element of $\pi(P)$ in $\Gamma(G)$,
we infer that $P$ is nilpotent and hence we may assume that $P$ is an $r$-group. Passing to the quotion of $G$ by $\Phi(P)$, we replace $P$ with an elementary abelian group. Since $C_G(P)$ is a normal subgroup of $G$,
we see that $C_G(P)\leq K$. Since $P\leq Z(C_G(P))$, the Hall $\{r\}'$-subgroup of $C_G(P)$ is normal in $G$,
so we may assume that $C_G(P)=P$. Observe that $K/P$ is a $\{2,3\}$-group, so $G/P$ is an extension of a $\{2,3\}$-group by ${^2}E_6(2)$. Let $H$ be a subgroup in $G/P$ of minimal order such that there exists a normal $\{2,3\}$-subgroup $T$ of $H$ 
with $H/T\simeq {^2}E_6(2)$. We claim that $T\leq Z(H)$. Take a Sylow subgroup $R$ of $T$. Then $N_H(R)/N_T(R)\simeq H/T$ by Frattini's argument. Since $H$ is minimal, it follows that $N_H(R)=H$. Thus, $T$ is nilpotent. Take $U\in Syl_{19}(H)$ and
assume that $[R,U]\neq 1$. Then $RU$ acts by conjugation on $P$ and Lemma~\ref{l:hallhigman} yields that $19r\in\omega(G)$; a contradiction. Therefore $[R,U]=1$, and hence $C_H(R)\not\subseteq T$. Since $C_H(R)$ is normal in $H$
and $H/T$ is a simple group, we infer that $C_H(R)T=H$. Therefore, $C_H(R)/C_T(R)\simeq H/T$. Since $H$ is minimal, it follows that $C_H(R)=H$ and hence $T\leq Z(H)$, as claimed. According to \cite{Atlas}, the group ${^2}E_6(2)$ includes a subgroup isomorphic to $O_8^-(2)$. Take a subgroup $M$ of $H$ with $M/T\simeq O_8^-(2)$. We claim that $M$ includes a subgroup isomorphic to $O_8^-(2)$. The minimal subgroup $M^{\infty}=\bigcap\limits_{n\geq0}M^{(n)}$ in the derived series of $M$ is a perfect group and hence $TM^{\infty}/T\simeq{O}_8^-(2)$. Therefore, $M^{\infty}/(T\cap M^{\infty})\simeq{O}_8^-(2)$. Since $T\cap M^{\infty}\leq Z(M^{\infty})$, we see that $|T\cap M^{\infty}|$ divides the order of the Schur multiplier of $O_8^-(2)$ and hence $T\cap M^{\infty}=1$. Thus, $M^{\infty}$ is the required group. 
Now $O_8^-(2)$ acts on $P$ and by Lemma~\ref{l:semidir} we may assume that the extension of $O_8^-(2)$ by $P$ splits. Observe that $7\in\pi(O_8^-(2))$ and $11\not\in\pi(O_8^-(2))$. 
From the character table of $O_8^-(2)$ in the case $r=11$ and the table of 7-modular characters of $O_8^-(2)$, when
$r=7$, we derive that $17r\in\omega(G)$; a contradiction. The proof of Lemma~\ref{l:S is alt} is complete.
\end{proof}

\section{Proof of the Corollary}
\begin{proof}
Suppose that $G$ is a finite group such that $\Gamma(G)=\Gamma(A_n)$ or $\Gamma(G)=\Gamma(S_n)$ for some integer $n\geq19$ and $n-p\leq4$, where $p$ is the largest prime less than or equal to $n$. By Theorem~1, there exists a normal subgroup $K$ of $G$ such that $A_t\leq G/K\leq S_t$ where $t\geq{p}$ and $p\not\in\pi(K)$. Since $5p\not\in\omega(G)$, we obtain $t\leq p+4$. Assume that $K$ is nonsolvable and $S$ is a nonabilean composition factor of $K$.
Then there exists $r\in\pi(S)$ with $r\geq5$. Lemma~\ref{l:Vakula} imply that $p\not\in\Aut(S)$. 
Finally, by Lemma~\ref{l:primes_in_K}, we infer that $rp\in\omega(G)$; a contradiction. Thus, $K$
is solvable and the corollary is established.
\end{proof}

{\it Acknowledgements}. The authors wish to thank Andrey Vasili$'$ev  for valuable comments on the manuscript and Gerhard Paseman for pointing out Dusart's paper.

\bibliographystyle{amsplain}

\begin{thebibliography}{10}

\bibitem{13Gor} I.B.\,Gorshkov, Recognizability of alternating groups by spectrum, Algebra and Logic, {\bf52}:1 (2013), 41--45.

\bibitem{14Gor} I.B.\,Gorshkov, Recognizability of symmetric groups by spectrum, Algebra and Logic, 
{\bf53}:6 (2015), 450--457.

\bibitem{16GorGrish}  I.B.\,Gorshkov, A.N.\,Grishkov, On recognition by spectrum of symmetric groups, 
Sib. Electron. Math. Rep., {\bf13} (2016), 111--121.

\bibitem{17Star} A.M.\,Staroletov, On recognition of alternating groups by prime graph, Sib. Electron. Math. Rep., {\bf14} (2017), 994--1010.

\bibitem{10Vakula} I.A.\,Vakula, On the structure of finite groups isospectral to an alternating group, Proc. Steklov Inst. Math. (Suppl.), 272, suppl. 1 (2011), 271--286.

\bibitem{Vas05} A.V.\,Vasil$'$ev, On connection between the structure of a finite group and the properties of its prime graph, Sib. Math. J., 2005, Vol.~46, N3, 396--404.

\bibitem{15Vas} A.V.\,Vasil$'$ev, On finite groups isospectral to simple classical groups, J. Algebra, 2015, Vol.~423, 318--374.

\bibitem{Zs} K.\,Zsigmondy, Zur Theorie der Potenzreste, Monatsh. Math. Phys., {\bf 3} (1892) 265--284.

\bibitem{piFunc} J.\,B.\,Rosser, L.\,Schoenfeld, Approximate formulas for some functions of prime numbers, Illinois J. Math., {\bf6}~(1) (1962), 64--94.

\bibitem{Dusart} P.\,Dusart, Explicit estimates of some functions over primes, Ramanujan J., {\bf} (2016) 1--25.

\bibitem{PrimeTable} D.N.\,Lehmer, List of prime numbers from 1 to 10006721, Washington, DC: Carnegie Institution, 1914. 

\bibitem{05VasVd.t}  A.V.\,Vasil$'$ev and  E.P.\,Vdovin, An adjacency criterion for the prime graph of a finite simple group, Algebra and Logic, {\bf44}:6 (2005), 381--406.

\bibitem{PrimeGap} L.J.\,Lander and T.R.\,Parkin, On first appearance of prime differences, Math. Comp., {\bf21} (1967), 483--488.

\bibitem{Zavar} A.V.\,Zavarnitsine, Finite simple groups with narrow prime spectrum, Sib. Electron. Math. Rep., {\bf6} (2009), 1--12.

\bibitem{11VasVd.t}  A.V.\,Vasil$'$ev and E.P.\,Vdovin, Cocliques of maximal size in the prime graph of a finite simple group, Algebra and Logic, {\bf50}:4 (2011), 292--322.

\bibitem{Atlas} J.H.\,Conway, R.T.\,Curtis, S.P.\,Norton, R.A.\,Parker, R.A.\,Wilson, Atlas of finite groups, Clarendon Press, Oxford, 1985.    

%\bibitem{GAP} The GAP~Group, GAP -- Groups, Algorithms, and Programming, Version 4.8.5 (2016), %\verb+(http://www.gap-system.org)+.

\end{thebibliography}

\Addresses

\end{document}